\documentclass[reqno,a4paper]{amsart}
\usepackage[english,activeacute]{babel}
\usepackage{amssymb,amsmath,amsthm,amsfonts,mathrsfs,scalefnt,graphicx,graphics,color,hyperref}
\usepackage[font=footnotesize]{caption}
\usepackage{lmodern}
\usepackage{textcomp}
\usepackage{tikz}
\usetikzlibrary{trees}
\usetikzlibrary{shapes}
\usepackage[on]{auto-pst-pdf}
\usepackage{pst-tree}
\usepackage{pst-all}
\usepackage{tabularx}
\usepackage[T1]{fontenc}
\usepackage{url}  
\usepackage[latin1]{inputenc}
\usepackage[foot]{amsaddr}
\theoremstyle{plain}
\newtheorem{theorem}{Theorem}[section]
\newtheorem{lemma}[theorem]{Lemma}
\newtheorem{proposition}[theorem]{Proposition}
\newtheorem{corollary}[theorem]{Corollary}

\newtheorem{remark}[theorem]{Remark}
\numberwithin{equation}{section}

\newcommand{\Rb}  {{\mathbb R}}

\newcommand{\Zb}  {{\mathbb Z}}

\newcommand{\As} {{\mathcal A}}

\newcommand{\Cs} {{\mathcal C}}

\newcommand{\Ns} {{\mathcal N}}

\DeclareMathOperator{\Var}{Var}

\renewcommand{\phi}{\varphi}

\newcommand{\ind}{1\!\kern-1pt \mathrm{I}}
\newcommand{\rsto}{]\!\kern-1.8pt ]}
\newcommand{\lsto}{[\!\kern-1.7pt [}

\newcommand\F{\mbox{I\kern-2pt F}}

\title[The deterministic limit of the Moran model]{The deterministic limit of the Moran model:\\ a uniform central limit theorem}
\author[F.~Cordero]{Fernando Cordero$^1$}
\address{$^1$Faculty of Technology, University of Bielefeld, Universit\"{a}tsstrasse 25, 33615 Bielefeld, Germany. E-mail: fcordero@techfak.uni-bielefeld.de}

\date{\today}%
\begin{document}
\begin{abstract}
We consider a Moran model with two allelic types, mutation and selection. In this work, we study the behaviour of the proportion of fit individuals when the size of the population tends to infinity, without any rescaling of parameters or time. We first prove that the latter converges, uniformly in compacts in probability, to the solution of an ordinary differential equation, which is explicitly solved. Next, we study the stability properties of its equilibrium points. Moreover, we show that the fluctuations of the proportion of fit individuals, after a proper normalization, satisfy a uniform central limit theorem in $[0,\infty)$. As a consequence, we deduce the convergence of the corresponding stationary distributions.
\vspace{.2cm}

\end{abstract}
\maketitle
\section{Introduction}
One of the main goals in population genetics is to describe the evolution of a population subject to the evolutionary forces of mutation and selection. Significant work has been done in this direction (for review see \cite{CK70,Bu98,Du08,Eth}). Among a large variety of stochastic and deterministic models, finite population models are considered to be more realistic. However, it is not always easy to deal with them. This situation leads to consider large population approximations.

In this paper we consider a Moran model with population size $N$, two allelic types, mutation, and selection. Rather than looking to its classical diffusion limit approximation, we aim here to study the behaviour of the proportion of fit individuals when the size of the population tends to infinity, without any rescaling of parameters or time. The theory of density dependent families of Markov chains provide us with the necessary tools to achieve this goal (see \cite{Ku70,Ku76, Ku81,Ku86}).

In a first step, as an application of \cite[Theorem 3.1]{Ku70}, we show a dynamical law of large numbers, which tell us that the proportion of fit individuals converges, uniformly in compacts in probability, to the solution of a ordinary differential equation. The latter is explicitly solved and an analysis of the stability of its equilibrium points is given. We also establish the connection of the deterministic limit with the well-known deterministic $2$-allele parallel mutation-selection model (see \cite{CK70}). Next, we turn out to the study of the fluctuations of the proportion of fit individuals around its deterministic limit. More precisely, based on \cite[Theorem 2.7]{Ku76} and \cite[Theorem 8.5]{Ku81}, we prove that, after a proper normalization, these fluctuations behave asymptotically as a Gaussian process. This result takes the form of a uniform central limit theorem in $[0,\infty)$. In \cite[Theorem 8.5]{Ku81}, the corresponding Gaussian process is characterised as the solution of a stochastic differential equation, which in our setting can be explicitly solved. Finally, as a consequence of the previous results, we prove that the underlying stationary distributions converge to a Dirac mass at the unique stable equilibrium point of the aforementioned ordinary differential equation. 

The paper is organized as follows. In Section \ref{s2}, we shortly describe the $2$-type Moran model with selection and mutation and we recall its basic properties. In Section \ref{s3}, we state and prove our main results, which deal with the asymptotic behaviour of the proportion of fit individuals: (1) a law of large numbers - Proposition \eqref{lln}, (2) a uniform central limit theorem - Theorem \eqref{pctt}, and (3) the convergence of the stationary distributions - Corollary \eqref{csd}.

\section{\texorpdfstring{Some basic facts on the $2$-type Moran model}{}}\label{s2}
The $2$-type Moran model of size $N$ with selection and mutation describes the evolution in (continuous) time of a population of size $N$ in which each individual is characterised by a type $i\in\{0,1\}$. The underlying dynamic is given as follows. If an individual reproduces, its single offspring  inherits the parent's type and replaces a uniformly chosen individual, possibly its own parent. The replaced individual dies, keeping the size of the population constant. 
 
Individuals of type $1$ reproduce at rate $1$, whereas individuals of type $0$ reproduce at rate $1+s_N$, $s_N\geq0$. Mutation occurs independently of reproduction. An individual of type $i$ mutates to type $j$ at rate $u_N\,\nu_j$, $u_N\geq 0$, $\nu_j\in(0,1)$, $\nu_0+\nu_1=1$.

We denote by $X_t^N$ the number of individuals of type $0$ at time $t$. The process $X^N:=(X_t^N)_{t\geq 0}$ is a continuous-time Markov chain with infinitesimal generator 
\begin{equation}\label{igen1}
\As_{X^N}f(k):=q_{k,k+1}^N\,\left(f(k+1)-f(k)\right)
+q_{k,k-1}^N\,\left(f(k-1)-f(k)\right),
\end{equation}
where
\begin{equation*}
 q_{k,k+\ell}^N:=\left\{ \begin{array}{ll}
                             k\,\frac{(N-k)}{N}\,(1+s_N) +(N-k)\,u_N\,\nu_0& \textrm{if $\ell=1$},\\
                            &\\
                             k\,\frac{(N-k)}{N} +k \,u_N\,\nu_1& \textrm{if $\ell=-1$},\\
                             &\\
                             0 & \textrm{if $|\ell|>1$},
                         \end{array}\right.
\end{equation*}
and $q_{k,k}^{N}:=-q_{k,k+1}^{N}-q_{k,k-1}^{N}$. In other words, $X^N$ is a birth-death process with birth rates $\lambda_k^N:=q_{k,k+1}^N$ and death rates $\mu_k^N:=q_{k,k-1}^N$.
In particular, when $u_N>0$, $X^N$ admits a unique stationary distribution, which is given by
\begin{equation*}
\pi_{X^N}(k):=C_N\prod_{i=1}^k \frac{\lambda_{i-1}^N}{\mu_i^N},\quad k\in\{0,\dots,N\}.
\end{equation*}
where $C_N$ is a normalising constant (see \cite{Du08}). When $u_N=0$, by contrast, $X^N$ is an absorbing Markov chain with $0$ and $N$ as absorbing states.

It is well known that, when the parameters of selection and mutation satisfy
\begin{equation}\label{dlas}
 \lim_{N\rightarrow\infty}Nu_N=\theta\in(0,\infty)\quad\textrm{and}\quad\lim_{N\rightarrow\infty}Ns_N=\sigma\in(0,\infty),
\end{equation}
the rescaled process $(X^N_{Nt}/N)_{t\geq 0}$ converges in distribution to the Wright-Fisher diffusion (see, e.g., \cite[p. 71, Lemma 5.11]{Eth}). By the latter we mean the continuous-time Markov process $Y:=(Y_t)_{t\geq 0}$ with infinitesimal generator given by 
$$\As_Y f(x):=x(1-x)\frac{d^2f}{dx^2}(x)+ \left(\sigma(1-x)x+\theta\nu_0(1-x)-\theta \nu_1x\right)\frac{df}{dx}(x),\quad x\in[0,1].$$

\section{The deterministic limit}\label{s3}
\subsection{A law of large numbers}\label{s3.0}
In contrast to the diffusion limit framework, where it is assumed that the parameters of the model satisfy \eqref{dlas}, we consider here constant parameters of mutation and selection, i.e. $u_N=u\geq 0$ and $s_N=s\geq 0$. In addition, we do not rescale the time. In this setting, a deterministic limit emerges when the size of the population converges to infinity. To see this, we first observe that the infinitesimal parameters of $X^N$ satisfy
$$q_{k,k+\ell}^N=N q\left(\frac{k}{N},\ell\right),\quad \ell\neq 0,$$
where $q:\Rb\times\Zb\setminus\{0\}\rightarrow \Rb$ is defined by
\begin{equation*}
q(p,\ell):=\left\{ \begin{array}{ll}
                             (1+s)\,p(1-p) +u\,\nu_0\,(1-p)& \textrm{if $\ell=1$},\\
                             
                             p(1-p)+u\,\nu_1\,p& \textrm{if $\ell=-1$},\\
                             
                             0 & \textrm{if $|\ell|>1$}.
                         \end{array}\right.
\end{equation*}
We conclude that the sequence of Markov chains $(X^N)_{N\geq 1}$ is density dependent in the sense of \cite[Sect. 3]{Ku70} (see also \cite{Ku76}). Let us denote by $Z^N:=(Z_t^N)_{t\geq 0}$ the continuous-time Markov chain given by
$$Z_t^N:=\frac{1}{N}X_t^N,\quad t\geq 0.$$
In addition, we set $F:=q(\cdot,1)-q(\cdot,-1)$, i.e.
$$F(x):=sx(1-x)+u\nu_0(1-x)-u\nu_1 x=-sx^2+(s-u)x+u\nu_0,\quad x\in \Rb.$$
The next result provides the asymptotic behaviour of the process $Z^N$ in the form of a dynamical law of large numbers.
\begin{proposition}[Law of large numbers]\label{lln}
For each $z_0\in[0,1]$, we denote by $z(z_0,\cdot)$ the solution of
 \begin{align}
z(0)&=z_0\in[0,1]\quad\textrm{and}\quad\frac{dz}{dt}(t)=F(z(t)),\quad t\in[0,T].\label{deteq}
\end{align} 
Assume that $\lim_{N\rightarrow\infty}Z_0^N=z_0\in[0,1]$. Then, for all $\varepsilon>0$, we have
 \begin{equation*}
\lim\limits_{N\rightarrow\infty} P\left(\sup\limits_{t\leq T}|Z_t^N-z(z_0,t)|>\varepsilon\right)=0,
\end{equation*}
i.e. $Z^N$ converges to the solution of \eqref{deteq} uniformly in compacts in probability (convergence ucp).
\end{proposition}
\begin{proof}
Note that the function $F$ is Lipschitz in $[0,1]$ and that $F(0)=u\nu_0\geq 0$ and $F(1)=-u\nu_1\leq 0$. We conclude that, for all $z_0\in[0,1]$, Eq. \eqref{deteq} has a unique solution $z(z_0,\cdot)$ defined in $[0,\infty)$, which in addition does not leave the interval $[0,1]$.  

Besides that, $q(x,\ell)=0$ when $|\ell|>1$, and hence 
\begin{equation}\label{cond1}
 \sup_{x\in[0,1]}\sum_{\ell}|\ell|q(x,\ell)<\infty\quad\textrm{and}\quad \lim_{d\rightarrow\infty}\sup_{x\in[0,1]}\sum_{|\ell|>d}|\ell|q(x,\ell)=0.
\end{equation}
The desired result follows therefore as an application of \cite[Theorem 3.1]{Ku70} (see also \cite[Theorem 11.2.1]{Ku86}).
\end{proof}
\begin{remark}
An analogous result has been shown in \cite{Da14} for a discrete-time version of the Moran model. 
\end{remark}
Note that when $s=u=0$, $F$ is identically equal to $0$. Thus, the solution of \eqref{deteq} is constant. When $s=0$ and $u>0$, Eq. \eqref{deteq} is linear 
and has a unique equilibrium point given by $\nu_0$. In the remaining case, Eq. \eqref{deteq} is a Riccati equation with constant coefficients and can hence be solved by means of quadratures (see \cite{Ya14}). Moreover, \eqref{deteq} has two equilibrium points, $x_0^-<0<x_0^+<1$, given by the zeros of $F$, i.e.
\begin{equation}\label{xo+-}
 x_0^-:=\frac{s-u-\sqrt{\Delta}}{2s}\quad\textrm{and}\quad x_0^+:=\frac{s-u+\sqrt{\Delta}}{2s},
\end{equation}
where $\Delta:=(s-u)^2+4su\nu_0$. For simplicity, when $s=0$ and $u>0$, we set $x_0^+:=\nu_0$ and $\Delta:=u$.

The following lemma gives the explicit expression of the solution of \eqref{deteq} and the stability properties of its equilibrium points.
\begin{lemma}\label{expsol}
For each $z_0\in[0,1]$, the solution of \eqref{deteq} is given by
 \begin{equation*}
z(z_0,t):=\left\{ \begin{array}{ll}
                             z_0& \textrm{if $s=u=0$},\\
                             &\\
                             \nu_0 +(z_0-\nu_0)\,e^{-ut}& \textrm{if $s=0$ and $u>0$},\\
                             &\\
                             \frac{x_0^+(z_0-x_0^-)-x_0^-(z_0-x_0^+)e^{-s\left(x_0^+-x_0^-\right)t}}{(z_0-x_0^-)-(z_0-x_0^+)e^{-s\left(x_0^+-x_0^-\right)t}}& \textrm{if $s>0$},
                         \end{array}\right.
\end{equation*}
for all $t\geq 0$. Moreover, when $s+u>0$, the equilibrium point $x_0^+$ is asymptotically stable. More precisely,
 \begin{equation}\label{at}
 \lim\limits_{t\rightarrow\infty}z(z_0,t)=x_0^+.
\end{equation}
In addition, when $s>0$, the equilibrium point $x_0^-$ is unstable.
\end{lemma}
\begin{proof}
 The first part of the statement is obtained through a straightforward verification. When $s+u>0$, \eqref{at} is obtained from a direct calculation of the corresponding limit, or simply noting that
\begin{equation}\label{as}
\frac{dF}{dx}(x_0^+)=-\Delta<0.
 \end{equation}
 Similarly, if $s>0$, $\frac{dF}{dx}(x_0^-)=\sqrt{\Delta}>0$. The last assertion follows.
\end{proof}
\subsection{The underlying deterministic mutation-selection model}
In this section, we shortly describe the well-known deterministic $2$-allele parallel mutation-selection model and its relation with Eq. \eqref{deteq}. This model goes back to \cite[p. 265]{CK70} and is usually obtained via direct deterministic modelling, rather than via a limiting procedure starting from a stochastic model.

The deterministic $2$-allele parallel mutation-selection model describes the evolution of a population in which each individual is characterised by a type $i\in\{0,1\}$. The absolute frequencies of individuals of type $0$ and $1$ satisfy the following system of differential equations 
\begin{equation}\label{ldeteq}
\frac{dy}{dt}(t)=y(t)A,\quad\textrm{where}\quad A:=\begin{pmatrix}
     1+s-u\nu_1&u\nu_1\\u\nu_0&1-u\nu_0\\
    \end{pmatrix}.
\end{equation}
It is straightforward to see that, under the initial condition $y(0):=(z_0,1-z_0)$ with $z_0\in[0,1]$, the proportion of individuals of type $0$ evolves following  Eq. \eqref{deteq}. Therefore, we have
\begin{equation}\label{propy}
z(z_0,t)=\frac{y_0(t)}{y_0(t)+y_1(t)}. 
\end{equation}
Similarly, Eq. \eqref{deteq} can be turned into the linear system of differential equations \eqref{ldeteq} using the following transformation (see \cite{TMB74}, \cite{BW01} or \cite{BG07})
\begin{equation*}
 y(t):=\left(y_0(t),\,  y_1(t)\right)=\exp{\left(t+s\int\limits_0^t z(u)du\right)}\times\,\left(z(t),\,  1-z(t)\right).
\end{equation*}
We point out that this correspondence between the solutions of \eqref{deteq} and \eqref{ldeteq} gives an alternative way to obtain Lemma \ref{expsol}. We provide the details in the case $s>0$, the other cases are obtained in a similar way. We first express $y$ as
\begin{equation*}
y(t)=y(0)\,e^{At}.
\end{equation*}
Additionally, the eigenvalues of the matrix $A$ are $\lambda_+:=1+sx_0^+$ and $\lambda_-:=1+sx_0^-,$
with corresponding right eigenvectors given, up to multiplicative constants, by
\begin{equation*}
v_+:=\begin{pmatrix}
     u\nu_0+sx_0^+\\u\nu_0\\
    \end{pmatrix}\quad\textrm{and}\quad v_-:=\begin{pmatrix}
     u\nu_0+sx_0^-\\u\nu_0\\
    \end{pmatrix}.
\end{equation*}
Thus, the matrix $e^{At}$ can be explicitly computed through diagonalization of the matrix $A$. In particular, with the initial condition $y(0)=(z_0,\,1-z_0)$, we obtain
\begin{align*}
y_0(t)&=\frac{e^{(1+sx_0^+) t}}{x_0^+-x_0^-}\left(x_0^+(z_0-x_0^-)-x_0^-(z_0-x_0^+)e^{-s(x_0^+-x_0^-)t}\right),\\
y_1(t)&=\frac{e^{(1+sx_0^+) t}}{x_0^+-x_0^-}\left((1-x_0^+)(z_0-x_0^-)-(1-x_0^-)(z_0-x_0^+)e^{-s(x_0^+-x_0^-)t}\right).
\end{align*}
Plugging this in \eqref{propy} yields \eqref{expsol}. This alternative way of solving equation \eqref{deteq} has the advantage that it can be easily generalized to the multi-type case.
\subsection{\texorpdfstring{A uniform central limit theorem in $[0,\infty)$}{}}\label{s3.3}
We assume in the sequel that $u>0$. In this setting, it would be natural to study the behaviour of the stationary distribution of $Z^N$, denoted by $\pi_{Z^N}$, when $N$ tends to infinity. The convergence ucp of $Z^N$ to $z$, derived in Section \ref{s3.0}, is not sufficient to deduce such a result. For this reason, we aim to provide a uniform convergence result on $[0,\infty)$ providing the convergence of the limiting distributions. The central limit theorem in \cite{Ku76} (see also \cite{Ku81,No74}) for density dependent families of Markov chains give us a way to achieve our task. To see this, we first define
$$g(x):= q(x,1)+q(x,-1)=-(2+s)x^2+ \left(2+s-u(\nu_0-\nu_1)\right)\,x+ u\nu_0,\quad x\in\Rb.$$
Next, we consider, for $z_0\in[0,1]$, the Gaussian diffusion $V^{z_0}:=(V_t^{z_0})_{t\geq 0}$ given by
$$V_t^{z_0}:=\left\{ \begin{array}{ll}
                           F(z(z_0,t))\int\limits_0^t \frac{\sqrt{g(z(z_0,v))}}{F(z(z_0,v))}\,dB_v  & \textrm{if $z_0\neq x_0^+$},\\
                             
                           \sqrt{g(x_0^+)}\, e^{-\sqrt{\Delta}\,t}\int\limits_0^t e^{\sqrt{\Delta}\,v}\,dB_v  & \textrm{if $z_0=x_0^+$},
                         \end{array}\right.$$
where $(B_t)_{t\geq 0}$ is a standard Brownian motion and $z(z_0,\cdot)$ is the solution of \eqref{deteq}. One can easily check that, for all $x\in[0,1]$, $g(x)\geq u(\nu_0\wedge\nu_1)>0$. Hence, $V^{z_0}$ is well defined.

Now, we introduce the following characteristic functions:
$$\psi^{N}(t,\theta):=E\left[e^{i\theta\sqrt{N}\left(Z_t^N-z(z_0,t)\right)}\right]\quad\textrm{and}\quad\psi(t,\theta):=E\left[e^{i\theta V_t^{z_0}}\right]\quad t\geq0,\,\theta\in\Rb.$$
\begin{theorem}[Central limit theorem]\label{pctt}
Assume that $\lim_{N\rightarrow\infty}\sqrt{N}(Z_0^N-z_0)=0$. Then $\sqrt{N}(Z^N-z(z_0,\cdot))\xrightarrow[N\rightarrow\infty]{(d)}V^{z_0}$. Moreover, we have
\begin{equation*}
\lim\limits_{N\rightarrow\infty}\sup_{t\geq 0}|\psi_N(t,\theta)-\psi(t,\theta)|=0.
 \end{equation*}
\end{theorem}
\begin{proof}
We focus only on the case $z_0\neq x_0^+$. The case $z_0=x_0^+$ is proven analogously. Note that, in addition to \eqref{cond1}, $F\in\Cs^1$ and
$\lim_{d\rightarrow\infty}\sup_{x\in[0,1]}\sum_{|\ell|>d}|\ell|^2q(x,\ell)=0.$
Then, we deduce from \cite[Theorem 8.2]{Ku81} (or \cite[Theorem 11.2.3]{Ku86}) that 
\begin{equation}\label{clta}
\sqrt{N}(Z^N-z(z_0,\cdot))\xrightarrow[N\rightarrow\infty]{(d)}\tilde{V}^{z_0},
\end{equation}
where $\tilde{V}^{z_0}$ is the solution of the stochastic differential equation
\begin{equation}\label{sde2}
\tilde{V}_t^{z_0}=W_{C_1(t)}-\widehat{W}_{C_{-1}(t)}+\int\limits_0^t\frac{dF}{dx}(z(z_0,v))\,\tilde{V}_v^{z_0}\,dv,
\end{equation}
with $W,\,\widehat{W}$ being independent Brownian motions, and $C_\ell(t):=\int_0^t q(z(z_0,v),\ell)dv$. Furthermore, $V^{z_0}$ solves the following stochastic differential equation (see \cite[Chap. 9, Ex. 2.7]{ReYor99}):
\begin{equation}\label{sde1}
V_t^{z_0}=\int\limits_0^t\sqrt{g(z(z_0,v))}\,dB_v+\int\limits_0^t\frac{dF}{dx}(z(z_0,v))\,V_v^{z_0}\,dv,\quad t\geq 0.
\end{equation}
The last terms of the right hand sides of \eqref{sde2} and \eqref{sde1} have the same form. We claim that 
$$\left(W_{C_1(t)}-\widehat{W}_{C_{-1}(t)}\right)_{t\geq 0}\overset{(d)}{=}\left(\int\limits_0^t\sqrt{g(z(z_0,v))}\,dB_v\right)_{t\geq 0}.$$
In this case, we deduce that $\tilde{V}^{z_0}\overset{(d)}{=}V^{z_0}$ (from the uniqueness of the solutions of \eqref{sde2} and \eqref{sde1}). Hence, the first statement follows from \eqref{clta}. In addition, since we have \eqref{as}, and $Z^N$ and $z(z_0,\cdot)$ do not leave the interval $[0,1]$, the remaining result is obtained from \cite[Theorem 8.5]{Ku81} (or \cite[Theorem 2.7]{Ku76}). 

Now we show the claim. The involved processes are local martingales with the same quadratic variation given by the deterministic function $\int_0^t g(z(z_0,v))dv$. In addition, $\int_0^\infty g(z(z_0,v))dv=\infty$. The claim follows as an application of the Dambis-Dubins-Schwartz theorem (see for ex. \cite[Chap. 5, Theorem 1.6]{ReYor99}).
\end{proof}
\begin{lemma}\label{chf}
The characteristic function of $V_t^{z_0}$ is given by 
\begin{equation}\label{chatff}
\psi(t,\theta)=\left\{ \begin{array}{ll}
                           \exp\left(-\frac{\theta^2}{2}\,F^2(z(z_0,t))\int\limits_{z_0}^{z(z_0,t)}\frac{g(y)}{F^3(y)}dy\right)  & \textrm{if $z_0\neq x_0^+$},\\
                             
                            \exp\left(-\frac{\theta^2}{4\sqrt{\Delta}}\, g(x_0^+)\left(1-e^{-2\sqrt{\Delta} t}\right)\right)  & \textrm{if $z_0=x_0^+$}.
                         \end{array}\right. 
\end{equation}
In particular, we have
$$\Var(V_t^{z_0})=\left\{ \begin{array}{ll}
                         F^2(z(z_0,t))\int\limits_{z_0}^{z(z_0,t)}\frac{g(y)}{F^3(y)}dy,  & \textrm{if $z_0\neq x_0^+$},\\
                             
                          \frac{g(x_0^+)}{2\sqrt{\Delta}}\left(1-e^{-2\sqrt{\Delta} t}\right)   & \textrm{if $z_0=x_0^+$}.
                         \end{array}\right.$$
\end{lemma}
\begin{proof}
We give the proof for $z_0\neq x_0^+$, the remaining case follows similarly. The second statement follows from the first one, evaluating at $\theta=0$ the second derivative of $\psi$ with respect to $\theta$.

Note that the stochastic integral appearing in the definition of $V^{z_0}$ is a local martingale with deterministic quadratic variation given by 
$\int_0^t g(z(z_0,v))/F^2(z(z_0,v))dv$. Since $\int_0^\infty g(z(z_0,v))/F^2(z(z_0,v))dv=\infty$, we deduce from the Dambis-Dubins-Schwartz theorem (see \cite[Chap. 5, Theorem 1.6]{ReYor99}) that
$$\left(\int_0^t \frac{\sqrt{g(z(z_0,v))}}{F(z(z_0,v))}dB_v\right)_{t\geq 0}\overset{(d)}{=}\left(\beta_{\int\limits_0^t\frac{g(z(z_0,v))}{F^2(z(z_0,v))}dv}\right)_{t\geq 0},$$
where $(\beta_t)_{t\geq 0}$ is a standard Brownian motion. The result is obtained using that $E[e^{i\lambda \beta_t}]=e^{-\frac{\lambda^2}{2}t}$.
\end{proof}
The main significance of Theorem \eqref{pctt} is that it provides a central limit theorem for the stationary distributions $\pi_{Z^N}$. We point out that such a result makes also part of \cite[Theorem 2.7]{Ku76} in the framework of density dependent families of Markov chains. However, in the latter there is a mistake in the variance of the corresponding limiting Gaussian distribution. For this reason, we provide a detailed proof to the next result. 
\begin{corollary}\label{csd}
Let $Z^N_\infty$ be a random variable distributed as $\pi_{Z^N}$, then
\begin{equation*}
 \sqrt{N}\left( Z^N_\infty-x_0^+\right)\xrightarrow[N\rightarrow\infty]{(d)}\Ns\left(0,\frac{g(x_0^+)}{2\,\sqrt{\Delta}}\right).
\end{equation*}
In particular, we have $\pi_{Z^N}\xrightarrow[N\rightarrow\infty]{w}\delta_{x_0^+}$.
\end{corollary}
\begin{proof}
Note first that, Lemma \ref{chf} implies that $\lim_{t\rightarrow\infty}\Var(V_t^{z_0}):=\frac{g(x_0^+)}{2\,\sqrt{\Delta}}$. In addition, the random variables $V_t^{z_0}$ are centred and Gaussian. We conclude that
\begin{equation}\label{e1}
 V_t^{z_0}\xrightarrow[t\rightarrow\infty]{(d)}V_{\infty}^{z_0},
\end{equation}
where $V_{\infty}^{z_0}\sim\Ns\left(0,\frac{g(x_0^+)}{2\,\sqrt{\Delta}}\right)$. In addition, from the irreducibility of $Z^N$, we also have 
\begin{equation}\label{e2}
Z_t^N\xrightarrow[t\rightarrow\infty]{(d)}Z_\infty^N,
\end{equation}
where $Z_\infty^N\sim\pi_{Z^N}$. 

We fix now $\varepsilon>0$. Theorem \ref{pctt} implies that, there is $N_0(\varepsilon)$ such that, for all $N\geq N_0(\varepsilon)$, $\sup_{t\geq 0}|\psi_N(t,\theta)-\psi(t,\theta)|\leq \varepsilon$. Therefore, we have for $N\geq N_0(\varepsilon)$
\begin{align}\label{e4}
 \left\lvert E\left[e^{i\theta\sqrt{N}\left(Z_\infty^N-x_0^+\right)}\right]-E\left[e^{i\theta V^{z_0}_\infty}\right]\right\lvert&\leq \left\lvert E\left[e^{i\theta\sqrt{N} Z_\infty^N}\right]-E\left[e^{i\theta\sqrt{N} Z_t^N}\right]\right\lvert\nonumber\\
 & +\left\lvert e^{i\theta\sqrt{N} x_0^+}-e^{i\theta\sqrt{N} z(z_0,t)}\right\lvert+\varepsilon\nonumber\\
 & +\left\lvert E\left[e^{i\theta V_t^{z_0}}\right]-E\left[e^{i\theta V_\infty^{z_0}}\right]\right\lvert.\nonumber
\end{align}
Using \eqref{e1}, \eqref{e2}, \eqref{at} and taking the limit when $t$ tends to infinity in the previous inequality,  we deduce that
$$\left\lvert E\left[e^{i\theta\sqrt{N}\left(Z_\infty^N-x_0^+\right)}\right]-E\left[e^{i\theta V^{z_0}_\infty}\right]\right\lvert\leq \varepsilon.$$
Since this holds for all $\varepsilon>0$, the result follows.

\end{proof}

\subsection*{Acknowledgements}
I would like to thank Ellen Baake and Tom Kurtz for stimulating and fruitful discussions. This project received financial support from the Priority Programme \textit{Probabilistic Structures in Evolution} (SPP 1590), which is funded by Deutsche Forschungsgemeinschaft.
\bibliographystyle{acm}
\bibliography{reference}
\end{document}